\documentclass[11pt]{amsart}

\usepackage{amsfonts}
\usepackage{amsmath,amsthm,latexsym,amssymb,graphicx}
\usepackage{graphicx,color,enumerate}

\numberwithin{equation}{section}

\newtheorem{theorem}{Theorem}
\newtheorem{lemma}{Lemma}

\newtheorem{proposition}{Proposition}

\newtheorem{remark}{Remark}

\numberwithin{theorem}{section}
\numberwithin{corollary}{section}
\numberwithin{lemma}{section}
\numberwithin{definition}{section}
\numberwithin{proposition}{section}
\numberwithin{remark}{section}

\newcommand{\dint}{\ds\int}
\newcommand{\ds}{\displaystyle}

\newcommand{\R}{\mathbb R}
\newcommand{\N}{\mathbb N}

\newcommand{\medint}{-\kern  -,375cm\int}
\newcommand{\diver}{\mathrm {div}}

\thanks{
$^{1}$ Dipartimento di Matematica e Applicazioni ``R. Caccioppoli'', Universit\`{a}
degli Studi di Napoli ``Federico II'', Complesso Monte S. Angelo, via Cintia
- 80126 Napoli, Italy; email: brandolini@unina.it; fchiacch@unina.it; cristina@unina.it.
\\
\indent $^2$ Institut Elie Cartan, UMR CNRS 7502, Universit\'e de Lorraine, F-54506 Vandoeuvre-les-Nancy Cedex, France; {antoine.henrot@univ-lorraine.fr}}

\begin{document}
\title[Dirichlet-Laplacian with a drift]{Existence of minimizers for eigenvalues of the Dirichlet-Laplacian with a drift}
\author[\textsf{B.~Brandolini}]{Barbara Brandolini$^{1}$}
\author[\textsf{F.~Chiacchio}]{Francesco Chiacchio$^{1}$}
\author[\textsf{A.~Henrot}]{Antoine Henrot$^{2}$}
\author[\textsf{C.~Trombetti}]{Cristina Trombetti$^{1}$}

\begin{abstract}
This paper deals with the eigenvalue problem for the operator $L=-\Delta -x\cdot \nabla $ with Dirichlet boundary conditions. We are interested in proving the existence of a set minimizing any eigenvalue $\lambda_k$ of $L$ under a suitable measure constraint suggested by the structure of the operator. More precisely we prove that for any $c>0$ and $k\in \N$ the following minimization problem
$$
\min\left\{\lambda_k(\Omega): \> \Omega \>\mbox{quasi-open} \>\mbox{set}, \> \int_\Omega e^{|x|^2/2}dx\le c\right\}
$$
has a solution. 
\end{abstract}

\maketitle

\section{Introduction}
In this paper we are interested in the following eigenvalue problem for the Dirichlet-Laplacian with a drift term
\begin{equation}  \label{problem1}
\left\{ 
\begin{array}{ll}
-\Delta u - x\cdot \nabla u=\lambda u & \mbox{in}\> \Omega \\ 
u=0 & \mbox{on}\> \partial\Omega,
\end{array}
\right.
\end{equation}
or equivalently in the weighted eigenvalue problem
\begin{equation}  \label{problem}
\left\{ 
\begin{array}{ll}
-\mathrm{div} \left(e^{|x|^2/2} \nabla u\right)=\lambda  e^{|x|^2/2}
u & \mbox{in}\> \Omega \\ 
u=0 & \mbox{on}\> \partial\Omega,
\end{array}
\right.
\end{equation}
where $\Omega$ is an open subset of $\R^N$ ($N \ge 2$). Let us denote
$$
dm_N=\prod_{i=1}^N e^{\frac{x_i^2}{2}}dx_i, \qquad x=(x_1,...,x_N) \in 
\mathbb{R}^N,
$$
and let $H_0^1(\Omega;m_N)$ be the closure of $C_0^\infty(\Omega)$ with respect to the norm 
$$
||u||_{H_0^1(\Omega;m_N)}=\left(||u||^2_{L^2(\Omega;m_N)}+||\nabla u||^2_{L^2(\Omega;m_N)}\right)^{1/2}.
$$
The operator $\mathcal{R}:f \in L^2(\Omega; m_N) \to \varphi \in H_0^1(\Omega; m_N)$, where $\varphi$ is the unique solution to 
\begin{equation}\label{1}
\left\{ 
\begin{array}{ll}
-\mathrm{div} \left(e^{|x|^2/2} \nabla \varphi \right)=f  e^{|x|^2/2}
 & \mbox{in}\> \Omega \\ 
\varphi=0 & \mbox{on}\> \partial\Omega,
\end{array}
\right.
\end{equation}
is compact, self-adjoint and nonnegative (see Section 2); then the spectrum of $\mathcal R$ is purely discrete, it consists only of eigenvalues which can be ordered (according to their multiplicity):
$$
0< \lambda_1(\Omega) \le \lambda_2(\Omega) \le ... \le \lambda_k(\Omega) \le ...
$$
Our main result is  the following
\begin{theorem} \label{existence} 
For any $c>0$ the minimum 
\begin{equation}\label{min}
\mathrm{min} \{ \lambda_k(\Omega): \> \Omega \subset \R^N, \> \Omega \> quasi-open\> set, m_N(\Omega) \leq c\}
\end{equation}
is achieved.
\end{theorem}
	
For the definition of quasi-open sets we remind the reader to Remark \ref{rem-qo1} and the references mentioned therein. Let us briefly discuss how our result is inserted in the literature. In the case of Laplace operator, the analogous minimization problem with Lebesgue measure constraint has been faced for the first time by  Buttazzo and Dal Maso in \cite{BD}. Their key assumption is that $\Omega$ varies in the class of sets contained in the same box $D$. Replacing $D$ with $\R^N$ is far from being simple due to the lack of compactness for generic sequences of sets. Very recently, this problem has been overcome independently by Mazzoleni and Pratelli in \cite{MP} and Bucur in \cite{B} with different techniques. In our case, the set $\Omega$ is allowed to vary in the whole $\R^N$ since the structure of the operator and the $m_N$ measure constraint allow us to earn the compact embedding of the weighted Sobolev space $H_0^1(\Omega;  m_N)$ into the weighted Lebesgue space $L^2(\Omega; m_N)$ (see Theorem \ref{compact} below). 

\noindent On the other hand problem \eqref{problem1} can be viewed as a prototype of a more general class of eigenvalue problems. For instance in \cite{HNR} (see also the references therein), among other things, the problem of minimizing the first eigenvalue of 
$$
\left\{
\begin{array}{ll}
-\diver\left(A(x)\nabla u\right)+v\cdot \nabla u+Vu=\lambda u &\mbox{in}\> \Omega \\
u=0 &\mbox{on}\>\partial\Omega
\end{array}
\right.
$$
  is addressed under various constraints on $A, v, V, \Omega$ by using a new notion of rearrangement. 
To our knowledge the existence of a domain minimizing a generic eigenvalue of problem \eqref{problem1} has not been established yet. In this paper we solve this question under the natural ``weighted volume constraint''.

Using an appropriate notion of rearrangement, see for instance  \cite{BMP,RCBM},  a Faber-Krahn type inequality can be proved:  the ball centered at the origin is the optimal domain for the first eigenvalue (i.e. the case $k=1$). Now all the other cases are open. In the classical situation (the Dirichlet-Laplacian with a constraint on the Lebesgue measure) only two cases are solved: for $k=1$, the minimizer is any ball (Faber-Krahn inequality), while for $k=2$, it is the union of two identical balls (Krahn-Sz\"ego inequality), see \cite{H} for more details. In our situation, even the case $k=2$ is not clear because of the measure $m_N$. In  \cite{BCHOT} we  study this problem and prove, among other things, that the optimal domain is not composed of two identical balls.

The proof of Theorem \ref{existence} can be summarized as follows. We consider a minimizing sequence of quasi-open sets $\Omega_n$ and we construct the sequence of functions $w_n\in H_0^1(\Omega_n;m_N)$ solving problem \eqref{1} in $\Omega_n$ with $f \equiv 1$. We prove that $w_n$ strongly converge to a function $w$ in $L^2(\R^N;m_N)$ and we define $\hat \Omega=\{w>0\}$. We prove that the eigenfunctions $u_j^n$ corresponding to $\lambda_j(\Omega_n)$ weakly converge to $u_j\in H_0^1(\hat \Omega;m_N)$. We conclude that $\lambda_j(\hat \Omega)$ is the minimum of problem \eqref{min}. The paper is organized as follows. In Section 2 we recall the compact embedding of $H_0^1(\R^N;m_N)$ into $L^2(\R^N;m_N)$ and we provide  an Hardy type inequality which in turn gives an improved embedding theorem. In Section 3 we prove a sharp reverse H\"older inequality for eigenfunctions that will be used to ensure the suitable convergence of $u_j^n$. Note that these results may have an interest by their own. Finally Section 4 contains the proof of Theorem \ref{existence}.

\section{Some properties of weighted Sobolev spaces}

\subsection{Weighted isoperimetric inequalities and rearrangements}

We start this section by recalling the isoperimetric inequality with respect to the measure $m_N$.  Let $\Omega\subset \R^N$ be a Lebesgue measurable set, we define the weighted perimeter of $\Omega$ with respect to $m_N$ by
$$
P_{m_N}(\Omega)=\sup\left\{\int_\Omega \diver \left({\bf k}(x)e^{|x|^2/2}\right)dx:\> {\bf k }\in C_0^1(\R^N,\R^N),\> |{\bf k}|\le 1\right\}.
$$
For any smooth set $\Omega \subset \R^N$ it reduces to
$$
P_{m_N}(\Omega)=\int_{\partial \Omega} e^{|x|^2/2}d\mathcal{H}^{N-1}.
$$ 
In \cite{BMP} (see also \cite{BCM,BCM1}) the authors prove the following result.

\begin{theorem}
For any set $\Omega \subset \R^N$ with finite $m_N$-measure,
\begin{equation}\label{II}
P_{m_N}(\Omega) \ge P_{m_N}(\Omega^\bigstar),
\end{equation}
where $\Omega^\bigstar$ is the ball centered at the origin, having the same $m_N$-measure as $\Omega$. Equality sign holds in \eqref{II} if and only if $\Omega=\Omega^\bigstar$.
\end{theorem}

\noindent As well-known, \eqref{II} turns out to be the key ingredient for a Faber-Krahn type inequality to hold (see Proposition \ref{faber}). To this aim we give the notion of rearrangement with respect to the measure $m_N$.

\noindent Let $\phi $ be a measurable real function defined in $\Omega$. The distribution function of $%
\phi $ with respect to the $m_N$-measure is defined by 
\begin{equation*}
\mu(t)=m_{N}\left( \{x\in \Omega:\>|\phi (x)|>t\}\right) ,\qquad t\geq 0,
\end{equation*}
while the decreasing rearrangement of $\phi $ with respect to the $m_N$-measure is the function 
\begin{equation*}
\phi ^{\ast }(s)=\sup \left\{ t\geq 0:\>\mu(t)>s\right\} ,\quad s\in (0,m_N(\Omega)).
\end{equation*}
It is easy to see that $\phi ^{\ast }$ is a nonincreasing, right-continuous
function defined in $(0,m_N(\Omega))$, equidistributed with $\phi $, that
means  $\phi $ and $\phi ^{\ast }$ have corresponding superlevel sets
with the same $m_N$-measure. This feature implies that 
\begin{equation*}
||\phi ||_{L^{p}(\Omega;m_N)}=||\phi ^{\ast }||_{L^{p}(0,m_N(\Omega))},\quad \forall p\geq 1.
\end{equation*} 
Now we set
\begin{equation*}
h(r)=N\omega _{N}e^{r^{2}/2}r^{N-1},\quad H(r)=\int_{0}^{r}h(t)dt,
\end{equation*}
where $\omega_N$ is the Lebesgue measure of the unit ball in $\R^N$. Then
$$
P_{m_N}(\Omega^\bigstar)=h\left(H^{-1}(m_N(\Omega))\right)
$$
and \eqref{II} reads as 
\begin{equation*}\label{II2}
P_{m_N}(\Omega)\ge h \left(H^{-1}(m_N(\Omega))\right).
\end{equation*}
We finally define $\phi ^{\bigstar }$, the $m_N$-symmetrization of $
\phi $, as follows
\begin{equation*}
\phi ^{\bigstar }(x)=\phi ^{\ast }\left( H(\left\vert x\right\vert )\right), \quad x \in \Omega^\bigstar.
\end{equation*}
$\phi^\bigstar$ is the only spherically symmetric function, nonincreasing along the radii, whose level sets are balls centered at the origin, with the same $m_N$ measure as the corresponding level sets of $|\phi|$. This definition immediately implies 
$$
||\phi||_{L^p(\Omega;m_N)}=||\phi^\bigstar||_{L^p(\Omega^\bigstar;m_N)},\qquad \forall p \ge 1.
$$
The following inequalities hold true.
\begin{proposition}[Hardy-Littlewood inequality]
Let $\phi,\psi \in L^2(\Omega; m_N)$; then
\begin{equation}\label{hl}
\int_\Omega |\phi \psi| dm_N \le \int_0^{m_N(\Omega)}\phi^\ast(s) \psi^\ast(s)ds=\int_{\Omega^\bigstar}\phi^\bigstar \psi^\bigstar dm_N.
\end{equation}
\end{proposition}
\begin{proposition}[P\'olya-Sz\"ego principle]\label{prop3.1}
If $\phi \in H^1(\R^N; m_N)$, $\phi \ge 0$, then $\phi^\bigstar \in H^1(\R^N;m_N)$ and
\begin{equation}\label{polya}
\int_{\R^N}|\nabla \phi|^2 dm_N \ge \int_{\R^N}|\nabla \phi^\bigstar|^2 dm_N.
\end{equation}
\end{proposition}

\noindent For an exhaustive treatment on rearrangements see, for instance, 
\cite{K,Ta,H,Ke,Ba}.

\medskip
\subsection{Weighted Sobolev spaces and embedding theorems}

In order to prove the existence of the optimal set in \eqref{min} let us introduce the natural Sobolev spaces associated with problem \eqref{problem}. Let $\Omega$ be an arbitrary open subset of $\R^N$; let us consider the weighted Lebesgue space
$$
L^q(\Omega;m_N)=\left\{u:\Omega\to \R: \> \int_\Omega |u|^q dm_N<+\infty\right\}, \qquad q \ge 1,
$$
endowed with the norm
$$
||u||_{L^q(\Omega;m_N)}=\left(\int_\Omega |u|^q dm_N\right)^{1/q},
$$
and let $H_0^1(\Omega;m_N)$ be the closure of $C_0^\infty(\Omega)$ with respect to the norm 
$$
||u||_{H_0^1(\Omega;m_N)}=\left(||u||^2_{L^2(\Omega;m_N)}+||\nabla u||^2_{L^2(\Omega;m_N)}\right)^{1/2}.
$$
We initially observe that, for any $\Omega \subseteq \R^N$, if $u \in H_0^1(\Omega; m_N)$ and we define $v=ue^{|x|^2/4}$, we get the following equivalence.

\begin{proposition}
$$u \in H_0^1(\Omega; m_N) \Longleftrightarrow v \in H_0^1(\Omega), \quad |x|v \in L^2(\Omega).$$
\end{proposition}

The following Poincar\'e inequality is well-known (see for instance \cite{EK}).
\begin{proposition}
For every $u \in H^1(\R^N;m_N)$ it holds
\begin{equation}
\label{poincare}
\int_{\R^N} |\nabla u|^2dm_N \ge N \int_{\R^N} u^2dm_N.
\end{equation}
\end{proposition}

The following  theorem provides the compact embedding of $H^1(\R^N; m_N)$ into $L^2(\R^N;m_N)$. Nevertheless this result can be found in \cite{EK}, for the reader's convenience we recall it here.

\begin{theorem}\label{compact}
The weighted Sobolev space $H^1(\R^N; m_N)$ is compactly embedded into the weighted Lebesgue space $L^2(\R^N;m_N)$.
\end{theorem}

\begin{proof}
Let $u_n \in H^1(\R^N; m_N)$ be such that 
\begin{equation}\label{111}
\int_{\R^N}|\nabla u_n|^2 dm_N \le C,\qquad n \in \N,
\end{equation}
and let $v_n=u_ne^{|x|^2/4}$. Integrating by parts we get
$$
\int_{\R^N}|\nabla v_n|^2dx+\frac{1}{4}\int_{\R^N}|x|^2v_n^2dx \le \int_{\R^N}|\nabla u_n|^2 dm_N
$$
and \eqref{111} immediately gives
\begin{eqnarray}
\int_{\R^N}|\nabla v_n|^2dx & \le C \label{a}
\\
\int_{\R^N}|x|^2v_n^2dx & \le C. \label{aa}
\end{eqnarray}
In order to prove the compactness of the sequence $\{v_n\}$ in $L^2(\R^N)$ it is enough  to show that
for any $\varepsilon >0$ there exist a constant $\delta >0$ and a set $D\subset \R^N$ such that
\begin{equation}
 \int_{\R^N} |v_n(x+\tau)-v_n(x)|^2 dx  <\varepsilon^2, \qquad \forall n \in \N, \> \forall \tau \in \R^N: \> |\tau|<\delta, \label{2}
\end{equation}
\begin{equation}
 \int_{\R^N \setminus \bar D} v_n^2 dx <\varepsilon^2, \qquad \forall n \in \N. \label{3} 
\end{equation} 
\eqref{2} is an immediate consequence of \eqref{a}. In order to prove \eqref{3} let us consider a ball $B_R$ centered at the origin, with radius $R$; by \eqref{aa} we get
$$
\int_{\R^N \setminus B_R} v_n^2 dx \le \frac{C}{R^2}
$$
and choosing $R$ in such a way that $\frac{C}{R^2}<\varepsilon^2$ we have \eqref{3}. Then, up to a subsequence, $v_n$ strongly converge to a function $v$ in $L^2(\R^N)$. Let $u=ve^{-|x|^2/4}$. Clearly $u \in L^2(\R^N; m_N)$ and 
$u_n$ strongly converge to $u$ in $L^2(\R^N; m_N)$.

\end{proof}

By the above result, as mentioned in the Introduction, the operator $\mathcal{R}:f \in L^2(\Omega; m_N) \to \varphi \in H_0^1(\Omega; m_N)$, where $\varphi$ is the unique solution to 
\begin{equation*}
\left\{ 
\begin{array}{ll}
-\mathrm{div} \left(e^{|x|^2/2} \nabla \varphi \right)=f  e^{|x|^2/2}
 & \mbox{in}\> \Omega \\ 
\varphi=0 & \mbox{on}\> \partial\Omega,
\end{array}
\right.
\end{equation*}
is compact; it is clearly self-adjoint and nonnegative, then the spectrum of $\mathcal R$ consists only of eigenvalues which can be ordered (according to their multiplicity):
$$
0< \lambda_1(\Omega) \le \lambda_2(\Omega) \le ... \le \lambda_k(\Omega) \le ...
$$
Moreover, for every $k \in \N$ and every $\lambda_k(\Omega)$ the following min-max formula holds 
\begin{equation}\label{minmax}
\lambda_k(\Omega)= \min_{\begin{array}{c}
V \mbox{ subspace of dim-}\\
\mbox{ension $k$ of } H_0^1(\Omega;m_N)
\end{array}} 
\max_{v\in V}\;\frac{\int_\Omega |\nabla u|^2 dm_N}{\int_\Omega u^2 dm_N}.
\end{equation}

Problem \eqref{min} is completely solved when $k=1$. Indeed, using  P\'olya-Sz\"ego inequality \eqref{polya} and the variational characterization of the first eigenvalue, arguing as for the Dirichlet-Laplacian, the following result can be easily proven.
\begin{proposition}\label{faber}
Let $\Omega$ be an open subset of $\R^N$ with finite $m_N$ measure. Then
$$
\lambda_1(\Omega) \ge \lambda_1(\Omega^\bigstar).
$$
\end{proposition}

\medskip
\subsection{An Hardy type inequality and consequences}

In \cite{EK} the authors, among other things, prove that, if $N \ge 3$ and $2^\ast=\frac{2N}{N-2}$, then
\begin{eqnarray*}
S&:=&\inf\left\{ \frac{\int_{\R^N}|\nabla \varphi|^2 dx}{\left(\int_{\R^N}|\varphi|^{2^\ast}dx\right)^{2/2^\ast}}:\> \varphi \in H^1(\R^N)\setminus\{0\} \right\} 
\\
& =&\inf\left\{ \frac{\int_{\R^N}|\nabla \varphi|^2 dm_N}{\left(\int_{\R^N}|\varphi|^{2^\ast}dm_N\right)^{2/2^\ast}}:\> \varphi \in H^1(\R^N;m_N)\setminus\{0\} \right\} 
\end{eqnarray*}
and, as a corollary, they get that for every $u \in H^1(\R^N;m_N)$
$$
\int_{\R^N}|\nabla u|^2 dm_N \ge S \left(\int_{\R^N}u^{2^\ast}dm_N\right)^{2/2^\ast}+\frac N 2 \int_{\R^N}u^2dm_N.
$$
Moreover, by interpolation between $L^2(\R^N;m_N)$ and $L^{2^\ast}(\R^N;m_N)$ they obtain that for any $2 \le q \le 2^\ast$ there exists a positive constant $C$ such that if $\frac 1 q= \frac a 2 + \frac{1-a}{2^\ast}$ and $u \in H^1(\R^N;m_N)$ then
$$
||u||_{L^q(\R^N;m_N)}\le C ||u||_{L^2(\R^N;m_N)}^a ||\nabla u||_{L^2(\R^N;m_N)}^{1-a}.
$$

We go further by proving a Hardy type inequality with respect to the measure $m_N$ (see for example \cite{T,BCT1}). This inequality, as in the classical case, will imply that, if $N \ge 3$,  
$H^1(\R^N;m_N)$ is continuously embedded into the weighted Lorentz space $L^{2^\ast,2}(\R^N;m_N)$  and a fortiori in $L^{2^\ast}(\R^N;m_N)$. When $N=2$ we gain that $H^1(\R^2;m_2)$ is continuously embedded into a suitable Orlicz space.

\noindent Let 
\begin{equation*}\label{rho}
\rho_N(r)=\frac{r^{1-N}e^{-r^2/2}}{\int_r^{+\infty}t^{1-N}e^{-t^2/2}dt}, \qquad r \ge 0.
\end{equation*}
Clearly
$$
\lim_{r \to 0^+} \rho_N(r)=+\infty,\qquad \lim_{r \to +\infty}\rho_N(r)=+\infty,
$$
and
$$
\lim_{r \to 0^+}r\rho_N(r)=N-2 \quad \mbox{if}\> N\ge 3, \quad \lim_{r\to 0^+}r\log \left(\frac 1 r\right)\rho_N(r)=1 \quad \mbox{if}\> N=2.
$$
Moreover $\rho_N$ solves the following differential equation
\begin{equation}\label{eqrho}
\rho_N'(r)+\frac{N-1}{r}\rho_N(r)+r\rho_N(r)=\rho_N^2(r);
\end{equation}
by differentiating \eqref{eqrho} we immediately get that $\rho_N$ cannot have positive maxima. Thus it has a unique minimum point $T>0$ where $\rho_N(T)>0$.

\begin{lemma}\label{unidim}
For every $\psi \in C_0^\infty(0,+\infty)$ it holds
\begin{equation*}
\int_0^{+\infty}(\psi'(r))^2 r^{N-1}e^{r^2/2}dr \ge \frac 1 4  \int_0^{+\infty}\psi(r)^2 \rho_N(r)^2r^{N-1}e^{r^2/2}dr.
\end{equation*}
Moreover $\frac 1 4$ is sharp.
\end{lemma}

\begin{proof}
Since $\left(\psi'(r)+\frac 1 2\rho_N(r)\psi(r)\right)^2 \ge 0$, integrating by parts we get
$$
\int_0^{+\infty}(\psi'(r))^2 r^{N-1}e^{r^2/2}dr \ge \frac 1 2 \int_0^{+\infty} \psi(r)^2 r^{N-1}e^{r^2/2} \left[\rho_N'(r)+\frac{N-1}{r}\rho_N(r)+r\rho_N(r)-\frac 1 2 \rho_N(r)^2 \right]dr
$$
and from \eqref{eqrho} we immediately deduce the claim.

In order to prove that the constant $\frac 1 4$ is sharp it suffices to consider the following sequence of functions
$$
\psi_k(r)=\left\{\begin{array}{ll}
\left(\dint_{1/k}^{+\infty}t^{1-N}e^{-t^2/2}dt\right)^{1/2} & r \in \left(0,\frac 1 k\right)
\\
\left(\dint_r^{+\infty}t^{1-N}e^{-t^2/2}dt\right)^{1/2} & r \in \left(\frac 1 k,+\infty\right)
\end{array}
\right.
$$
and verify that
$$
\lim_{k \to +\infty}\frac{\int_0^{+\infty}(\psi_k'(r))^2 r^{N-1}e^{r^2/2}dr}{\int_0^{+\infty}\psi_k(r)^2 \rho_N(r)^2r^{N-1}e^{r^2/2}dr}=\frac 1 4.
$$
\end{proof}

The following Hardy inequality holds true.
\begin{theorem}\label{hardy}
For every $u \in H^1(\R^N; m_N)$ it holds
\begin{equation}\label{hardy1}
\int_{\R^N} |\nabla u|^2 dm_N \ge \frac 1 4 \int_{\R^N} u^2 \rho_{N,T}(|x|)^2 dm_N,
\end{equation}
where 
$$
\rho_{N,T}(r)=\left\{\begin{array}{ll} 
 \rho_N(r) & 0<r<T 
 \\ 
 \rho_N(T) & r \ge T. 
 \end{array}    \right.
$$
Moreover $\frac 1 4 $ is sharp.
\end{theorem}

\begin{proof}[Proof of Theorem \ref{hardy}]
Let $u \in H^1(\R^N;m_N)$. Taking into account \eqref{polya} and \eqref{hl} it is enough to prove the claim when $u=u^\bigstar$. In this case
$$
\int_{\R^N}|\nabla u|^2 dm_N =N\omega_N\int_0^{+\infty} \left(u^\ast(H(r))'\right)^2r^{N-1}e^{r^2/2}dr
$$
and
$$
\int_{\R^N} u^2 \rho_{N,T}^2 dm_N=N\omega_N\int_0^{+\infty} u^\ast(H(r))^2 \rho_{N,T}(r)^2r^{N-1}e^{r^2/2}dr.
$$
We get \eqref{hardy1} by applying Lemma \ref{unidim} and $0 <\rho_{N,T}(r) \le \rho_N(r),\> r>0.$ 
\end{proof}

Let $N \ge 3$ and  $u \in H^1(\R^N; m_N)$. P\'olya-Sz\"ego principle \eqref{polya} together with Hardy inequality \eqref{hardy1} yield
\begin{eqnarray*}
 \int_{\R^N}|\nabla u|^2 dm_N \ge \int_{\R^N}|\nabla u^\bigstar|^2 dm_N &\ge& \frac 1 4  \int_{\R^N} \left(u^\bigstar\right)^2 \rho_{N,T}(|x|)^2dm_N
 \\
 &=& \frac 1 4  \int_0^{+\infty}u^\ast(t)^2 \rho_{N,T}(H^{-1}(t))^2 dt.
\end{eqnarray*}
Observing that $2/2^\ast-1=-2/N$ and
$$
\lim_{t \to 0^+} \frac{H^{-1}(t)}{t^{1/N}}=\omega_N^{-1/N}
$$
we get
$$
\int_{\R^N} |\nabla u|^2dm_N \ge C \int_0^{+\infty} u^\ast(t)^2t^{2/2^\ast-1}dt,
$$
that is $H^1(\R^N;m_N)$ is continuously embedded in the Lorentz space $L^{2^\ast,2}(\R^N;m_N)$ (see for instance \cite{Hu,KK} for the definition). On the other hand, when $N=2$ we obtain
$$
\int_{\R^N} |\nabla u|^2dm_N \ge C \int_0^{+\infty} \frac{u^\ast(t)^2}{t\max\{1,\log (1/t)\}^2}dt.
$$
By \cite{OP} (see Theorem 4.2 and Theorem 8.8) we deduce that  there exist $\gamma_0,\gamma_\infty>0$ such that
$$
\int_{0^+}\exp\left(\gamma_0 u^\ast(t)^2\right)dt<+\infty, \qquad \int^{+\infty}\exp\left(-\gamma_\infty u^\ast(t)^{-2}\right)dt<+\infty.
$$

\section{A reverse H\"older inequality}\label{sec3}

Let $u_{j}$ be an eigenfunction corresponding to the eigenvalue $\lambda _{j}
$ of the problem under consideration, i.e. 
\begin{equation}\label{P}
\left\{ 
\begin{array}{ll}
-\mathrm{div}\left( e^{|x|^{2}/2}\nabla u_{j}\right) =\lambda _{j}(\Omega
)e^{|x|^{2}/2}u_{j} & \mbox{in}\>\Omega \\ 
u_{j}=0 & \mbox{on}\>\partial \Omega,
\end{array}
\right. 
\end{equation}
where $\Omega$ is an open subset of $\mathbb{R}^{N}$ ($N \ge 2$) with finite $m
_{N}$-measure. The main result of this section is a sharp reverse H\"{o}lder
inequality for $u_{j}$. In the case of the Dirichlet-Laplacian, this kind of estimates has been  proved in \cite{PR,C}. 

The first step in our arguments consists into introducing  a ball $B_{\tilde r}$ such that $\lambda_j(\Omega)=\lambda_j(B_{\tilde r})$. Since the explicit value of $\lambda_j(B_{\tilde r})$ is not known, we estimate it from above and below in terms of $\tilde r$. To this aim observe that,  if $u_j$ is a solution to \eqref{P}, the function $v_j=u_je^{|x|^2/4}$ satisfies the following Dirichlet problem for the harmonic oscillator
\begin{equation}  \label{problem2}
\left\{ 
\begin{array}{ll}
-\Delta v_j +\frac{1}{4}|x|^2 v_j=\nu_j(\Omega) v_j & \mbox{in}\> \Omega \\ 
v_j=0 & \mbox{on}\> \partial\Omega
\end{array}
\right.
\end{equation}
with $\nu_j(\Omega) = \lambda_j(\Omega) -\frac{N}{2}$.

\noindent When $\Omega=\R^N$ the spectrum and the eigenfunctions of \eqref{problem2} are explicitly known (see for instance \cite{EK}). In particular the spectrum is given by $\{\nu_j(\R^N)=N+j-1:\>j=1,2,...\}$. When $j=1$, $\nu_1(\R^N)=N$ is simple and a corresponding eigenfunction is $e^{-|x|^2/2}$. 

\noindent When $\Omega \subset \R^N$, since the eigenvalues of \eqref{problem2} are decreasing with respect to inclusion of sets, we get 
\begin{equation}\label{RN}
\nu_j(\Omega) \ge \nu_1(\R^N)  \Longleftrightarrow \lambda_j(\Omega) \ge \frac{3}{2}N.
\end{equation}
Moreover, using $v_j$ as test function in \eqref{problem2} we get
\begin{equation*}
\lambda_j(\Omega)-\frac{N}{2}\ge \frac{\int_\Omega |\nabla v_j|^2 dx}{\int_\Omega v_j^2 dx}\ge \min\left\{\frac{\int_\Omega |\nabla \varphi|^2 dx}{\int_\Omega \varphi^2 dx}: \> \varphi \in H_0^1(\Omega)\setminus\{0\}\right\}=\lambda_1^{-\Delta}(\Omega)
\end{equation*}
where $\lambda_1^{-\Delta}(\Omega)$ is the first eigenvalue of the Dirichlet-Laplacian in $\Omega$. Whenever $\Omega$ has finite $m_N$-measure, by the well-known Faber-Krahn inequality, if $B_R$ is a ball with the same Lebesgue measure as $\Omega$, it holds
$$
\lambda_1^{-\Delta}(\Omega) \ge \lambda_1^{-\Delta}(B_R)=\frac{j_{N/2-1,1}^2}{R^2}
$$
being $j_{N/2-1,1}$ the first zero of the Bessel function of the first kind of order $\frac{N}{2}-1$. Thus
\begin{equation}\label{delta}
\lambda_j(\Omega)\ge \frac{N}{2}+\frac{j_{N/2-1,1}^2}{R^2}.
\end{equation}
Incidentally we note that 
$$
\lambda_1(\Omega) \le \frac{N}{2}+\lambda_1^{-\Delta}(\Omega)+\frac{R_\Omega^2}{4},
$$
where $R_\Omega$ is the radius of the smallest ball centered at the origin and containing $\Omega$.

Now, let us come back to problem \eqref{P}. From now on we will suppose that $\Omega$ has finite $m_N$-measure. By integrating the equation in \eqref{P} on the superlevel sets of $u_j$, using isoperimetric inequality \eqref{II}, co-area formula and H\"older inequality, according to a technique introduced by Talenti in \cite{Ta1}, 
in \cite{BMP} it is proved that 
\begin{equation*}
\left\{ 
\begin{array}{ll}
-U_{j}^{\prime \prime }(s)\leq \lambda _{j}(\Omega )I^{-2}(s)U_{j}(s) & 
\mbox{in}\>(0,m_{N}(\Omega)) \\ 
&  \\ 
U_{j}(0)=U_{j}^{\prime }(m_{N}(\Omega))=0, & 
\end{array}
\right. 
\end{equation*}
where
\begin{equation*}
U_{j}(s)=\int_{0}^{s}u_{j}^{\ast }(t)dt 
\end{equation*}
and 
$$I(s)=\inf\{P_{m_N}(E): \> E \>\mbox{smooth}, \> m_N(E)=s  \}=h\left(H^{-1}(s)\right)$$
is the isoperimetric function associated to the measure $m_N$.

\noindent For any fixed $L>0$ we consider the following Sturm-Liouville problem%
\begin{equation}
\left\{ 
\begin{array}{ll}
-\varphi^{\prime \prime }(s)=\sigma I^{-2}(s)\varphi(s) & \mbox{in}\>(0,L) \\ 
&  \\ 
\varphi(0)=\varphi^{\prime }(L)=0. & 
\end{array}
\right.   \label{SL}
\end{equation}
Since
\begin{equation*}
\lim_{s\rightarrow 0^{+}}\frac{I(s)}{s^{1-\frac{1}{N}}}=N\omega _{N}^{\frac{1%
}{N}}>0,
\end{equation*}
the Sobolev space $\{\varphi \in H^{1}(0,L):\varphi (0)=0\}$ is compactly
embedded into the weighted Lebesgue space 
\begin{equation*}
L^{2}\left((0,L);I^{-2}\right)=\left\{ \varphi :\left( 0,L\right) \rightarrow \mathbb{R}
:\>\left\Vert \varphi \right\Vert _{L^{2}\left((0,L);I^{-2}\right)}^{2}\equiv
\int_{0}^{L}(\varphi (t))^{2}I^{-2}(t)dt<+\infty \right\} 
\end{equation*}
(see \cite{N, BCT2}). Therefore spectral theory on selfadjoint compact operators
ensures that the first eigenvalue $\sigma _{1}(0,L)$ of (\ref{SL}) is simple
and it can be found as the minimum of the Rayleigh quotient 
\begin{equation*}
\min_{{\tiny 
\begin{array}{l}
\varphi \in H^{1}(0,L), \\ 
\varphi (0)=0,\>\varphi \not\equiv 0
\end{array}%
}}\frac{\displaystyle\int_{0}^{L}\left( \varphi ^{\prime }(t)\right) ^{2}dt}{
\displaystyle\int_{0}^{L}\left( \varphi (t)\right) ^{2}I^{-2}(t)dt}.
\end{equation*}

\bigskip

\noindent Now we claim that there exists a value of $L,$ say $\widetilde{L}$ , such
that 
\begin{equation*}
\sigma _{1}(0,\widetilde{L})=\lambda _{j}(\Omega).  \label{L}
\end{equation*}
To this aim consider the problem
\begin{equation}
\left\{ 
\begin{array}{ll}
-\mathrm{div}\left( e^{|x|^{2}/2}\nabla z\right) =\lambda e^{|x|^{2}/2}z & 
\mbox{in}\>B_{r} \\ 
z=0 & \mbox{on}\>\partial B_{r},
\end{array}
\right.   \label{Radial_P}
\end{equation}
where $B_{r}$ denotes the ball centered at the origin having radius $r$. By Theorem \ref{compact} the first eigenvalue $\lambda _{1}(B_{r})$ of (\ref
{Radial_P}) fulfills
\begin{equation*}
\lambda _{1}(B_{r})=\min \left\{ \frac{\int_{B_{r}}|\nabla \psi |^{2}dm_N}{\int_{B_{r}}\psi ^{2}dm_N}:\>\psi \in
H_{0}^{1}(B_{r};m_N)\backslash \left\{ 0\right\} \right\} .
\end{equation*}
By \eqref{polya} we know that such a minimum is achieved on a function $z$
such that%
\begin{equation*}
z(x)=z^{\bigstar }(x).
\end{equation*}
At this point it easy to verify  that $\lambda _{1}(B_{r})$ is a continuous
and strictly decreasing function \ with respect to $r.$
Moreover an easy consequence of \eqref{RN}, \eqref{delta} is 
\begin{equation*}
\lim_{r\rightarrow 0^{+}}\lambda _{1}(B_{r})=+\infty \text{ \ and  }
\lim_{r\rightarrow +\infty }\lambda _{1}(B_{r})=\frac{3}{2}N.
\end{equation*}
Therefore there exists a unique value of $r,$ say $\widetilde{r},$ such that 
\begin{equation*}
\lambda _{1}(B_{\widetilde{r}})=\lambda _{j}(\Omega).
\end{equation*}%
Let us denote with $\widetilde{z}$ an eigenfunction corresponding to $
\lambda _{1}(B_{\widetilde{r}})$; it is easy to verify that the function
\begin{equation*}
Z(s)=\int_{0}^{s}\widetilde{z}^{\ast }(t)dt
\end{equation*}
satisfies \eqref{SL} with $\sigma _{1}(0,\widetilde{L})=\lambda_1(B_{\widetilde r})=\lambda _{j}(\Omega).$
Note that if $\widetilde L=m_N(\Omega)$ the results we are going to state
become trivial since in this case $U_{j}$ and $Z$ are proportional. So from
now on we will assume that $\widetilde L<m_N(\Omega)$ and we will define the
function $Z(s)$ on the whole interval $(0,m_N(\Omega))$ by setting
its value constantly equal to $Z(\widetilde L)$  on $(\widetilde L,m_N(\Omega))$.

The following comparison result holds true. We omit the proof, since it can be obtained following, for instance, the lines of \cite{AFT,BCT}.

\begin{proposition}
If $u_j$ and $\widetilde z$ are defined as above and 
\begin{equation*}
\int_{0}^{m_{N}(\Omega)}\left( u_{j}^{\ast }(t)\right)
^{q}dt=\int_{0}^{\widetilde L}\left( \widetilde z^{\ast }(t)\right) ^{q}dt\qquad \text{ with }q>0
\end{equation*}
then
\begin{equation*}
\int_{0}^{s}\left( u_{j}^{\ast }(t)\right) ^{q}dt\leq \int_{0}^{s}\left(
\widetilde z^{\ast }(t)\right) ^{q}dt, \qquad s \in (0,\widetilde L).
\end{equation*}
\end{proposition}

The above result immediately implies the following reverse H\"{o}lder inequality.
\begin{theorem}\label{ReverseHolder} 
For any $0<r<q\le \infty$, $u_j$ satisfies the inequality
\begin{equation}\label{chiti}
\left( \int_{\Omega}\left\vert u_{j}\right\vert ^{q}dm_{N}\right)
^{1/q}\leq C(N,r,q,\lambda_j(\Omega))\left( \int_{\Omega}\left\vert u_{j}\right\vert ^{r}dm_{N}\right) ^{1/r},
\end{equation}
where
$$
C(N,r,q,\lambda_j(\Omega))=\frac{\left( \dint_0^{\widetilde L}\widetilde z^\ast(t)^{q}dt\right)
^{1/q}}{\left( \dint_0^{\widetilde L} \widetilde z^\ast(t)^{r}dt\right)^{1/r}}
$$
with $\widetilde z$ defined as above.
\end{theorem}
\begin{remark}\label{rem-qo1} \rm 
The previous inequality is stated for any open set $\Omega$ with finite $m_N$-measure. Actually, it also holds
for any quasi-open set $\Omega$ with finite $m_N$-measure. To see that, we can use the definition of a quasi-open set
(see for example \cite[chapter 3]{HP}) to approach $\Omega$ by a sequence of open sets $\Omega_\varepsilon$ such that $\Omega\subset \Omega_\varepsilon$ and the capacity of the
difference $\Omega_\varepsilon\setminus\Omega$ is less than $\varepsilon$. Then $\Omega_\varepsilon$ $\gamma$-converges to $\Omega$ and therefore the eigenfunctions and the eigenvalues of $\Omega_\varepsilon$ converge to the corresponding eigenfunctions and eigenvalues of $\Omega$ allowing to pass to the limit in (\ref{chiti}).
\end{remark}

\section{Proof of Theorem \ref{existence}}

Let $\Omega_n$ be a minimizing sequence, and let $w_n$ be the solution to
\begin{equation}  \label{pb1}
\left\{ 
\begin{array}{ll}
-\Delta w_n - x\cdot \nabla w_n=1 & \mbox{in}\> \Omega_n \\ 
w_n=0 & \mbox{on}\> \partial\Omega_n.
\end{array}
\right.
\end{equation}
Choosing $w_n$ as test function in \eqref{pb1} and using Poincar\'e
inequality \eqref{poincare} we get that the sequence $w_n$ is bounded in $H^1(\mathbb{R}^N;
m_N)$. The compact embedding of $H^1(\mathbb{R}^N; m_N)$
in $L^2(\mathbb{R}^N; m_N)$  ensures the existence of
a subsequence, still denoted by $w_n$, and the existence of a function $w
\in H^1(\mathbb{R}^N; m_N)$ such that 
\begin{eqnarray}
& w_n \rightharpoonup w \quad &  \mbox{in } \quad H^1(\mathbb{R}^N; m_N) 
\notag \\
& w_n \rightarrow w &  \mbox{in} \quad L^2(\mathbb{R}^N; m_N) \notag \\
& w_n \rightarrow w &  \mbox{a.e. in } \Omega. 
\notag
\end{eqnarray}
Let us consider the quasi-open set 
\begin{equation}\label{ast}
\hat\Omega = \{x \in \mathbb{R}^N : w(x) > 0\}.
\end{equation}
Since $w_n$ converges a.e. to $w$, we have for a.e. $x \in \mathbb{R}^N$ 
\begin{equation*}
\chi_{\hat\Omega}(x) \leq \liminf_n \chi_{\Omega_n}(x);
\end{equation*}
therefore Fatou's Lemma gives 
\begin{equation*}  \label{lim1}
m_N(\hat\Omega)=\int_{\mathbb{R}^N} \chi_{\hat\Omega}(x) dm_N \leq \liminf_n \int_{\mathbb{R}^N} \chi_{\Omega_n}(x) dm_N \leq c.
\end{equation*}
We want to prove that 
$$\lambda_k(\hat\Omega)=\min \{ \lambda_k(\Omega):\> m_N(\Omega) \leq c\}.$$
Let $u_n^j$ be an eigenfunction corresponding to $\lambda_j(\Omega_n)$, $1 \le j\le k$, normalized as follows
\begin{equation}\label{uno}
\int_{\Omega_n} (u_n^j)^2dm_N=1.
\end{equation}
By Theorem \ref{compact} there exist $k$ function $u^j$ such that
\begin{eqnarray}
& u_n^j \rightharpoonup u^j \quad &  \mbox{in } \quad H^1(\mathbb{R}^N; m_N) 
\notag \\
& u_n^j \rightarrow u^j &  \mbox{in} \quad L^2(\mathbb{R}^N; m_N). \notag 
\end{eqnarray} 
First of all we observe that $u^j \in H_0^1(\hat\Omega;m_N)$ (see Proposition \ref{conv1} below). 
 
\noindent Let us consider the set $V = \mathrm{Span} [u^1, u^2,\cdots, u^k]$
which is a vector subspace of  $H_0^1(\Omega; m_N)$ with dimension $k$.

\noindent Let $v = \sum_{j=1}^{k} \alpha_j u^j \in V$ and $v_n = \sum_{j=1}^{k}
\alpha_j u_n^j \in H_0^1(\Omega_n; m_N)$; then we have
\begin{eqnarray}
v_n \rightharpoonup v & \quad \mbox{in } \quad H^1(\mathbb{R}^N; m_N) 
\notag \\
v_n \rightarrow v & \quad \mbox{in} \quad L^2(\mathbb{R}^N; m_N). 
\notag
\end{eqnarray}
Thus
\begin{equation}  \label{lim2}
\dfrac{\displaystyle\int_{\mathbb{R}^N} |\nabla v|^2 dm_N}{\displaystyle%
\int_{\mathbb{R}^N} v^2 dm_N} \leq \liminf_n\dfrac{\displaystyle\int_{\mathbb{R}
^N} |\nabla v_n|^2 dm_N}{\displaystyle\int_{\mathbb{R}^N} v_n^2
dm_N}
\end{equation}
with
\begin{equation*}
\int_{\mathbb{R}^N} v^2_n dm_N = \sum_{j=1}^{ k} \alpha_j^2; \quad
\int_{\mathbb{R}^N} |\nabla v_n|^2 dm_N = \sum_{j=1}^{ k} \alpha_j^2
\lambda_j(\Omega_n).
\end{equation*}
Being 
$$\dfrac{\alpha_1^2}{\displaystyle\sum_{j=1}^{ k} \alpha_j^2}\lambda_1(\Omega_n)+...+\dfrac{\alpha_k^2}{\displaystyle\sum_{j=1}^{ k} \alpha_j^2}\lambda_k(\Omega_n)\le \lambda_k(\Omega_n)$$
yields
\begin{equation*}
\dfrac{\displaystyle\int_{\mathbb{R}^N} |\nabla v_n|^2 dm_N}{%
\displaystyle\int_{\mathbb{R}^N} v_n^2 dm_N} \leq \lambda_k(\Omega_n)
\end{equation*}
and then, for every $v \in V$, by \eqref{lim2} we get
\begin{equation*}
\dfrac{\displaystyle\int_{\mathbb{R}^N} |\nabla v|^2 dm_N}{\displaystyle%
\int_{\mathbb{R}^N} v^2 dm_N} \leq \liminf_n \lambda_k(\Omega_n).
\end{equation*}
By the min-max formula \eqref{minmax} for $\lambda_k(\Omega)$ we have
\begin{equation*}
\lambda _{k}(\hat\Omega)\leq \max_{v\in V}\dfrac{\displaystyle\int_{\mathbb{R}
^{N}}|\nabla v|^{2}dm_N}{\displaystyle\int_{\mathbb{R} 
^{N}}v^{2}dm_N}\leq \liminf_{n}\lambda _{k}(\Omega _{n}),
\end{equation*}
which concludes the proof.

\begin{proposition} \label{conv1} 
Let $\Omega_n$ be a minimizing sequence in problem \eqref{min} and let $u_n^j$ be an eigenfunction associated to $
\lambda_j(\Omega_n),$ $1 \le j \le k$, satisfying \eqref{uno}. For every $1 \le j \le k$  there exists $u^j \in H_0^1(\hat\Omega; m_N)$, with $\hat\Omega$ as in \eqref{ast}, such that
\begin{eqnarray}
& u_n^j \rightharpoonup u^j \quad &  \mbox{in } \quad H^1(\mathbb{R}^N; m_N) 
\label{4} \\
& u_n^j \rightarrow u^j &  \mbox{in} \quad L^2(\mathbb{R}^N; m_N). \label{5} 
\end{eqnarray} 
\end{proposition}

\begin{proof}
We first observe that \eqref{4} and \eqref{5} are easy consequences of Theorem \ref{compact}. It remains to prove that $u^j \in H_0^1(\hat\Omega; m_N)$. By  \eqref{chiti} (see Remark \ref{rem-qo1}) there exists a constant $M>0$, whose value is independent  from $n$, such that 
$$
||u_n^j||_\infty\le M.
$$
Suppose $\lambda_k(\Omega_n)\le \Lambda$ for every $n$ and, as in the proof of Theorem \ref{existence}, let $w_n$ be the solution to problem \eqref{pb1}. The function $\psi_n^j=\Lambda M w_n-u_n^j$ satisfies
$$
\left\{
\begin{array}{ll}
-\Delta \psi_n^j-x\cdot \nabla \psi_n^j =f_n^j & \mathrm{in}\> \Omega_n
\\
\psi_n^j=0 & \mathrm{on}\> \partial \Omega_n
\end{array}
\right.
$$
with $f_n^j=\Lambda M -\lambda_j(\Omega_n)u_n^j$ positive and bounded. By the maximum principle (see Proposition \ref{maxprinc} below)
$$
\psi_n^j \ge 0 \quad \mathrm{in} \>\Omega_n,
$$
that is
$$
u_n^j \le \Lambda M w_n.
$$
Analogously we can prove that $-\Lambda M w_n \le u_n^j$, and then
$$
|u_n^j|\le \Lambda M w_n.
$$
Passing to the limit, we get
\begin{equation}\label{6}
|u^j|\le \Lambda M w \quad \mathrm{a.e.}.
\end{equation}
Since $w=0$ quasi-everywhere in $\R^N \setminus \hat\Omega$, \eqref{6} implies that $u^j \in H_0^1(\hat\Omega;m_N)$.

\end{proof}

\begin{proposition}\label{maxprinc}
Let $D$ be an open set in $\R^N$ with finite $m_N$-measure and let $\psi$ be a solution to
\begin{equation}\label{eq-mp}
\left\{
\begin{array}{ll}
-\Delta \psi -x \cdot \nabla \psi =f & \mathrm{in}\> D
\\
\psi=0 & \mathrm{on}\> \partial D
\end{array}
\right.
\end{equation}
with $f\in L^\infty(D)$. If $f \ge 0$ in $D$ then $\psi \ge 0$ in $D$.
\end{proposition}

\begin{proof}
Let $\eta=\psi e^{|x|^2/4}$. Clearly $\eta$ satisfies
$$
\left\{
\begin{array}{ll}
-\Delta \eta+\left(\frac{|x|^2}{4}+\frac{N}{2}\right)\eta=fe^{-|x|^2/4} & \mathrm{in}\> D
\\
\eta=0 & \mathrm{on}\> \partial D.
\end{array}
\right.
$$
Since $f\in L^\infty(D)$, by Talenti's theorem (see \cite{Ta1}) $\eta \in L^\infty(D)$. Let us consider the sequence of sets $D_k=D \cap B_k$, where $B_k$ is the ball centered at the origin, with radius $k$; it holds
$$
D_1 \subset D_2\subset ...,\qquad D \subseteq \cup_{k\in \N} D_k,\qquad \partial D_k=\Gamma_k \cup \Gamma_k' \quad \mathrm{with}\> \Gamma_k \subseteq \partial D, \Gamma_k' \subset D.
$$
Let
$$
\Phi_N(t)=\left\{
\begin{array}{ll}
\ln |t| & \mbox{if}\> N=2
\\ \\
-|t|^{2-N} & \mbox{if}\> N\ge 3.
\end{array}
\right.
$$
We distinguish two cases: $(i) \>0 \notin \bar D$, $(ii) \>0 \in \bar D$.

\noindent $(i)$ Let $r_0>0$ be such that $|x|>r_0$ for every $x \in \bar D$. For any $k\in \N$ such that $k>r_0$ we define 
$$
w_k(x)=\left\{
\begin{array}{ll}
\Phi_N(|x|)-\Phi_N(r_0),\quad x\in D_k
\\ \\
\Phi_N(k)-\Phi_N(r_0),\quad x\in D\setminus D_k
\end{array}
\right.
$$
and
$$
w(x)=\Phi_N(|x|)-\Phi_N(r_0),\quad x\in D.
$$
By construction $w_k(x) \le w(x)$ for every $x \in D$ and 
$$\limsup_k\left(\inf_{\Gamma_k'}\frac{\eta}{w_k}\right)=\limsup_k\left(\frac{1}{\Phi_N(k)-\Phi_N(r_0)}\inf_{\Gamma_k'}\eta\right)=0.$$
 Moreover it satisfies
$$
\left\{
\begin{array}{ll}
-\Delta w_k+\left(\frac{|x|^2}{4}+\frac{N}{2}\right)w_k\ge 0 & \mathrm{in}\> D_k
\\
w_k>0 & \mathrm{on}\> D_k \cup \partial D_k.
\end{array}
\right.
$$
By Theorem 19, p. 97 in \cite{PW} we get that $\eta \ge 0$ and hence $\psi \ge 0$ in $D$.

$(ii)$ If $0 \notin \bar D$ it is enough to consider $\tilde D_k=D_k\setminus B_{2\varepsilon}(0)$ for $\varepsilon >0$ and 
$$
\tilde w_k(x)=\left\{
\begin{array}{ll}
\Phi_N(|x|)-\Phi_N(\varepsilon),\quad x\in \tilde D_k
\\ \\
\Phi_N(k)-\Phi_N(\varepsilon),\quad x\in D\setminus D_k.
\end{array}
\right.
$$
Reasoning as in the case $(i)$ we get that $\psi \ge 0$ in $D \setminus B_{2\varepsilon}(0)$ for every $\varepsilon >0$ small enough. By continuity $ \psi \ge 0$ in D.

\end{proof}
\begin{remark}\rm 
The previous maximum principle also holds
for a quasi-open set $D$ with finite $m_N$-measure. To see that, we proceed by external approximation $D_\varepsilon$, exactly as in Remark \ref{rem-qo1}, and we use the fact that the (non-negative) solutions to problem  (\ref{eq-mp}) on $D_\varepsilon$ converge to the solution of the same problem on $D$,
thus this one is non negative.
\end{remark}

\end{document}